\documentclass[12pt,twoside]{article}
\title{Gorenstein  dimensions in trivial ring extensions}
\date{}
\usepackage{amsfonts}
\usepackage{amsmath}
\usepackage{amssymb}
\usepackage{latexsym}
\usepackage{graphicx}

\newtheorem{thm}{\bf Theorem}[section]
\newtheorem{cor}[thm]{\bf Corollary}
\newtheorem{lem}[thm]{\bf Lemma}
\newtheorem{prop}[thm]{\bf Proposition}

\newtheorem{rem}[thm]{\bf Remark}

\newtheorem{exmp}[thm]{\bf Example}

\catcode`\ç=13
\defç{\c{c}}
\catcode`\é=13
\defé{\'e}
\catcode`\à=13
\defà{\`a}
\catcode`\è=13
\defè{\`e}
\catcode`\â=13
\defâ{\^a}
\catcode`\ù=13
\defù{\`u}
\catcode`\ê=13
\defê{\^e}
\catcode`\î=13
\defî{\^\i}
\catcode`\ô=13
\defô{\^o}
\newcommand{\field}[1]{\mathbb{#1}}

\newcommand{\Q }{\field{Q}}
\newcommand{\Z }{\field{Z}}

\def\gldim{{\rm gldim}}

\def\Gwdim{{\rm G\!-\!wdim}}
\def\Ggldim{{\rm G\!-\!gldim}}
\def\gldim{{\rm \!gldim}}
\def\wgldim{{\rm \!w.gl.dim}}

\def\pd{{\rm pd}}
\def\fd{{\rm fd}}
\def\id{{\rm id}}

\def\Gpd{{\rm Gpd}}
\def\Gfd{{\rm Gfd}}
\def\Gid{{\rm Gid}}

\def\Ext{{\rm Ext}}
\def\Tor{{\rm Tor}}
\def\Hom{{\rm Hom}}

\def\sup{{\rm sup}}

\def\qf{{\rm qf}}

\def\1{{\noindent\rm (1)}}
\def\2{{\noindent\rm (2)}}
\def\3{{\noindent\rm (3)}}
\def\4{{\noindent\rm (4)}}
\def\5{{\noindent\rm (5)}}
\begin{document}
\thispagestyle{empty}

\maketitle \vspace*{-1.5cm}
\begin{center}{\large\bf  Najib Mahdou and Khalid Ouarghi}

\bigskip

 \small{Department of Mathematics, Faculty of Science and Technology of Fez,\\ Box 2202, University S. M.
Ben Abdellah Fez, Morocco\\[0.12cm]
mahdou@hotmail.com\\
ouarghi.khalid@hotmail.fr }
\end{center}

\bigskip\bigskip

\noindent{\large\bf Abstract.}In this paper,   we show that the
Gorenstein global dimension of  trivial ring extensions is often
infinite. Also we study the transfer of Gorenstein properties
between a ring and its trivial ring extensions.
  We conclude with an example showing that, in
general, the transfer of the notion of Gorenstein projective
module does not carry up to pullback constructions.\bigskip

\bigskip\bigskip


\begin{section}{Introduction}

  Throughout this work, all rings are  commutative  with identity
                element and  all modules are unital. Let $R$ be a ring and $M$  an $R$-module. We use
                $\pd_R(M)$, $\id_R(M)$  and $\fd_R(M)$ to denote the usual projective, injective and flat
                dimensions of M, respectively. It is convenient to use ``local'' to
                refer to (not necessarily Noetherian) rings with a unique maximal
                ideal.

                In 1967-69, Auslander and Bridger \cite{A1,A2} introduced the concept of
                G-dimension for finitely generated modules over Noetherian rings.
                Several decades later, Enochs,
                Jenda and Torrecillas \cite{EJ1,EJ2,EJT} extended this notion by introducing three homological dimensions called
                Gorenstein projective, injective, and flat dimensions, which have
                all been studied extensively by their founders and also by Avramov,
                Christensen, Foxby, Frankild, Holm, Martsinkovsky, and Xu among
                others \cite{Avr,LW,CFH,GdimCM,HH,Xu}. For a ring $R$, the Gorenstein projective, injective and flat
                dimension of an $R$-module $M$ denoted $\Gpd_R(M)$, $\Gid_R(M)$ and $\Gfd_R(M)$, respectively, is defined
                in terms of resolutions of
                Gorenstein projective, injective and flat modules, respectively (see \cite{HH}). The Gorenstein projective dimension is a
                refinement of projective dimension to the effect that $\Gpd_R(M) \leq \pd_R(M)$ and equality holds when $\pd_R(M)$ is
                finite.\\
                           Recently, in \cite{BM}, the authors
                           introduced three classes of modules
                           called strongly Gorenstein projective,
                           injective and flat modules. These modules allowed for nice characterizations
            of Gorenstein projective and injective modules \cite[Theorem 2.7]{BM},
            similar to the characterization of projective modules
            via the free modules. In  \cite{BM2}, the authors  started the study of
                      Gorenstein homological dimensions
                     of a ring $R$; namely,
                     the Gorenstein global
                     dimension of $R$, denoted $\Ggldim(R)$, and the Gorenstein weak (global)
                     dimension of $R$, denoted $\Gwdim(R)$, and defined as follows:
$\Ggldim(R)=\sup\{\Gpd_R(M)\,|\,M\;R\!-\!module\}
=\sup\{\Gid_R(M)\,|\,M\;R\!-\!module\}$ \cite[Theorem 3.2]{BM2}
and $\Gwdim(R) =\sup\{\Gfd_R(M) \mid M\; R\!-\!module\}$. They
proved that, for any ring R, $\Gwdim(R) \leq \Ggldim(R)$
                        \cite[Theorems 4.2]{BM2} and that
                          the Gorenstein
                        weak and global dimensions are refinements
                        of the classical ones, i.e.,
                        $\Ggldim(R) \leq \gldim(R)$ and $\Gwdim(R)\leq \wgldim(R)$ with quality holding
                        if the weak global dimension of $R$ is finite \cite[Propositions 3.11 and
                        4.5]{BM2}.\\
                        This paper studies the Gorenstein dimensions in trivial ring extensions. Let $A$ be a ring and $E$ an $A$-module. The trivial ring extension of
                       $A$ by $E$ is the ring $R := A \ltimes E$ whose underlying group is
                       $A \times E$ with multiplication given by $(a,e)(a',e') = (aa',ae'
                       +a'e)$ \cite{Huckaba,KabMah}. Specifically, we investigate the possible transfer of Gorenstein
                       properties between a ring $A$ and its trivial ring extensions. Section 2 deals with the
                       descent and ascent of the (strongly) Gorenstein
                       properties between   $A$-modules and $R$-modules, where $R$ is a trivial ring
                       extension of $A$ (Theorem~\ref{thmtransSG-proj}, Corollary~\ref{cortransGP} and Proposition~\ref{lem1-2}). The last part of this section
                       is dedicated to  the Gorenstein global
                       dimension (Theorem~\ref{thm1-1}). In Section 3, we compute $\Ggldim(A\ltimes E)$ when
                       $(A,m)$ is a local ring with $mE=0$ (Theorem~\ref{thm1}) as well as  $\Ggldim(D\ltimes E)$ when $D$
                       is an integral domain and $E$ is an
                       $\qf(D)$-vector space (Theorem~\ref{thm2}).
                       The last theorem gives rise to an example
                       showing that, in general,
                       the notion of Gorenstein projective
                       module does not carry up to pullback
                       constructions (Example~\ref{exppulb}).

\end{section}

\begin{section}{Transfer of Gorenstein properties to trivial ring extensions}\label{sec:2}

Throughout this section, we adopt the following notation:  $A$ is
a ring, $E$
  an $A$-module and
$R=A\ltimes E$, the trivial ring extension of $A$ by  $E$. We
study  the transfer of (strongly) Gorenstein projective and
injective notions  between  $A$ and  $R$. We start this section
with the following theorem which handles the transfer of strongly
Gorenstein properties between  $A$-modules and $R$-modules.

\begin{thm}\label{thmtransSG-proj}
 Let $M$ be an $A$-module. Then:
 \begin{enumerate}
    \item
    \begin{description}
        \item[\rm (a)] Suppose that $\pd_A(E)< \infty$. If $M$ is a strongly
    Gorenstein projective $A$-module, then $M\otimes_A R$ is a
    strongly Gorenstein projective $R$-module.
        \item[\rm (b)] Conversely, suppose that $E$ is a flat $A$-module. If $M\otimes_A R$ is
    a strongly Gorenstein projective $R$-module, then $M$ is a strongly
    Gorenstein projective $A$-module.
    \end{description}
    \item Suppose that $\Ext_A^p(R,M)=0$ for all $p\geq 1$ and $\fd_A(R)<\infty$. If $M$ is a strongly
    Gorenstein injective $A$-module, then $\Hom_A(R,M)$ is a strongly
    Gorenstein injective $R$-module.
 \end{enumerate}
\end{thm}

\begin{proof}(1) (a) Suppose that $M$ is a strongly Gorenstein projective
$A$-module. Then  there is an exact sequence of $A$-modules:
$$ 0\rightarrow M\rightarrow
P\rightarrow M\rightarrow 0  \qquad(\star)$$
 where $P$ is projective \cite[Proposition 2.9]{BM}. It is known that $R=A\oplus_{A} E$ and since
$\pd_A(E)< \infty$ we have $\pd_A(R)< \infty$ and from the exact
sequence $(\star)$, $\Tor_A^i(M,R)=0$, $\forall\, i\geq 1$. Then
the sequence $0\rightarrow M\otimes_A R\rightarrow P\otimes_A
R\rightarrow M\otimes_A R\rightarrow 0$ is exact. Note that
$P\otimes_A R$ is a projective $R$-module. On the other hand, for
any $R$-module projective $Q$, $\pd_A(Q)< \infty$ \cite[Exercise
5, page 360]{CE}. Then, since $M$ is strongly Gorenstein
projective, $\Ext_R(M\otimes_A R, Q)=\Ext_A(M, Q)=0$ \cite[page
118]{CE}. Therefore $ M\otimes_A R$ is a strongly Gorenstein
projective $R$-module \cite[Proposition 2.9]{BM}.

(b) If $E$ is a flat $A$-module, then $R=A\ltimes E$ is a
faithfully flat $A$-module. Suppose that $M\otimes_A R$ is
strongly Gorenstein projective;  combining \cite[ Remark 2.8]{BM}
and \cite[ Proposition 2.9]{BM}, there is an exact sequence of $
R$-modules:
 $$0\rightarrow M\otimes_A R\rightarrow F\rightarrow
M\otimes_A R\rightarrow 0\qquad(\star\star)$$
 where $F= R^{(J)}$
is a free $R$-module. Then the  sequence $(\star\star)$ is
equivalent to the exact sequence:
$$0\rightarrow M\otimes_A R\rightarrow A^{(J)}\otimes_A
R\rightarrow M\otimes_A R\rightarrow 0.$$
 Since $R$ is a
faithfully flat $A$-module, the sequence of $A$-module
$0\rightarrow M\rightarrow A^{(J)}\rightarrow M\rightarrow 0$ is
exact. On the other hand, let $P$ be a projective $A$-module. Then
$P\otimes_A R$ is a projective $R$-module and
$\Ext_A^k(M,P\otimes_A R)=\Ext_R^k(M\otimes_A R, P\otimes_A R)=0$,
since $\Tor_i^A(M,R)=0$ and by \cite[Proposition 4.1.3, page
118]{CE}. But $0=\Ext_A^k(M,P\otimes_A R)\cong
\Ext_A^k(M,P)\oplus_A \Ext_A^k(M,P\otimes_A E)$, then
$\Ext_A^k(M,P)=0$. Therefore $M$ is a strongly Gorenstein
projective $A$-module.

(2) If $M$ is a strongly   Gorenstein injective $A$-module, there
exists an exact sequence of $A$-modules:
$$0\rightarrow M\rightarrow
I\rightarrow M\rightarrow 0$$
 where $I$ is an injective $A$-module. Since $\Ext_A(R,M)=0$, the
sequence
$$0\rightarrow \Hom_A(R,M)\rightarrow
\Hom_A(R,I)\rightarrow \Hom_A(R,M)\rightarrow 0$$
 is exact.  Note
that $\Hom_A(R,I)$ is an injective $R$-module. On the other hand,
for any   injective $R$-module $J$, we have $\id_A(J)<\infty$
(since $\fd_A(R)<\infty$ and by \cite[Exercise 5, page 360]{CE})
and $\Ext_R^i(J,\Hom_A(R,M))\cong \Ext_A^i(J,M)=0$
\cite[Proposition 4.1.4, page 118]{CE}. Therefore $\Hom_A(R,M)$ is
a strongly Gorenstein injective $R$-module.
\end{proof}

\begin{rem}\label{rem1}
The statements (1)(a) and (b) in Theorem~\ref{thmtransSG-proj}
hold for any homomorphism from $A$ to $R$  of finite projective
dimension  in (a) and faithfully flat  in (b), respectively. But
here we restrain our study to trivial ring extensions.
\end{rem}

\begin{cor}\label{cortransGP}
 Let $M$ be an $A$-module. Then:

\begin{enumerate}
    \item
 Suppose that $\pd_A(E)< \infty$. If $M$ is
    a Gorenstein projective $A$-module, then $M\otimes_A R$ is a
     Gorenstein projective $R$-module.

    \item Suppose that $\Ext_A^p(R,M)=0$ for all $p\geq 1$ and $\fd_A(R)<\infty$. If $M$ is
    a Gorenstein injective $A$-module, then $\Hom_A(R,M)$ is a
    Gorenstein injective $R$-module.
\end{enumerate}

\end{cor}

Next we compare the Gorenstein projective (resp., injective)
dimension of an $A$-module $M$ and the Gorenstein projective
(resp., injective) dimension of $M\otimes_{A} R$ (resp.,
$\Hom_A(R,M)$) as an $R$-module.

\begin{prop}\label{lem1-2}

 Let  $M$ be an $A$-module. Then:

\begin{enumerate}
    \item Suppose that $\Tor^{k}_A(M,R)=0$,
$\forall\;k\geq 1$. Then: $$\Gpd_A(M)\leq \Gpd_R(M\otimes_{A} R)
.$$
    \item Suppose that $\Ext^{k}_A(R,M)=0$,
$\forall\;k\geq 1$. Then: $$\Gid_A(M)\leq \Gid_R(\Hom_A(R,M)).$$
\end{enumerate}
\end{prop}

\begin{proof}

(1) By hypothesis  $\Tor^k_A (M,R) = 0$ for all $k \geq 1$. So, by
\cite[Proposition 4.1.3, page 118]{CE}, for any $A$-module $P$ and
all $n \geq 1$ we have
$$\Ext^k_A(M,P\otimes_AR) \cong \Ext^k_R(M \otimes_A R,P\otimes_A R).$$
Suppose that $\Gpd_R(M\otimes_{A} R)\leq d$ for some integer
$d\geq 0$. Let $P$ be a projective $A$-module. Then by
\cite[Theorem 2.20]{HH}, $0=\Ext_R^{d+1}(M\otimes_{A}
R,P\otimes_{A} R)\cong \Ext^{d+1}_A(M,P\otimes_AR)$. But
$0=\Ext^{d+1}_A(M,P\otimes_AR)\cong \Ext^{d+1}_A(M,P)\oplus
\Ext^{d+1}_A(M,P\otimes_A E)$. So $\Ext^{d+1}_A(M,P)=0$ for any
projective $A$-module  $P$. Therefore $\Gpd_A(M)\leq d$.

(2) The proof is essentially dual to (1). Here we use
\cite[Proposition 4.1.4, page 118]{CE} instead of
\cite[Proposition 4.1.3, page 118]{CE}.\end{proof}

The following thm gives a relation between  $\Ggldim(A)$ and
$\Ggldim(R)$.

\begin{thm}\label{thm1-1}
Suppose that $\Ggldim(A)$ is finite and $\fd_A(E)=r$, for some
integer $r\geq 0$. Then:
$$\Ggldim(A)\leq \Ggldim(R)+r.$$
\end{thm}

\begin{proof}

 Let $M$ be an $A$-module and let
$$\;P_{r}\stackrel{f_{r}}{\rightarrow}P_{r-1}\stackrel{f_{r-1}}
{\rightarrow}
 ...{\rightarrow}P_0\stackrel{f_{0}}{\rightarrow} M\rightarrow 0\;\;\;(i)$$
 be an exact sequence of $A$-modules where each $P_i$ is
 projective. Since $R=A\oplus_A E$ as an $A$-modules, $\fd_A(E)=\fd_A
 (R)=r$. Then for all $k\geq 1$ we have
 $$\;\Tor_{A}^{k}(Imf_r , R)\cong \Tor_{A}^{k+r} (M,R)=0\;\;\;\;(ii)$$
   If $\Ggldim(R)\leq n$, then $\Gpd_R(Imf_r\otimes_{A} R)\leq
 n$,  and by Proposition~\ref{lem1-2} we have $\Gpd_A(Imf_r)\leq n$. From
 (i), we have
 $0=\Ext_A^{n+1}(Imf_r,P)\cong \Ext_A^{r+n+1}(M,P)$, for every projective $A$-module $P$. Therefore  $\Gpd_A(M)\leq n+r$
 and so $\Ggldim(A)\leq
 n+r$ \cite[Lemma 3.3]{BM2}.
 \end{proof}

\end{section}
\begin{section}{Gorenstein global dimension of some trivial ring extensions}\label{sec:3}

   In this section, we study the Gorenstein global
            dimension of particular trivial ring extensions.
            We start by investigating  the Gorenstein  global
            dimension of  $R=A\ltimes E$, where $(A,m)$ is
            a local ring with maximal ideal $m$ and $E$ is an
            $A$-module such that $mE=0$. Recall that a  Noetherian ring $R$ is
            quasi-Frobenius if $\id_R(R)=0$ and a ring $R$ is
            perfect if all flat $R$-modules are projective \cite{Rot}.

   Next we announce the first main result of this section.

\begin{thm}\label{thm1}

Let $(A,m)$ be a local  ring
                    with maximal ideal $m$ and  $E$ an $A$-module
                    such that
                    $mE=0$. Let $R=A\ltimes E$. Then:
                    \begin{enumerate}
                        \item If $A$ is  a Noetherian ring  which is not  a field and  $E$ is  a finitely
                        generated  $A$-module
                        (i.e., $R$ is Noetherian), then
                        $\Ggldim{(R)}=\infty$.
                        \item If $A$ is a perfect ring, then
                        $\Ggldim{(R)}=$ either $\infty$ or $0$. Moreover, in the case $\Ggldim{(R)}=0$, necessarily
                       $A=K$ is a field and $E$ is a $K$-vector space with $dim_k E=1$ (i.e., $R=K\ltimes K$).
                    \end{enumerate}
\end{thm}

        To prove this thm, we need the following Lemmas.

\begin{lem}[{\cite[Lemma 3.4]{BM2}}]\label{lem2.1}
        Let $R$ be  a ring with $\Ggldim(R)<\infty$ and let $n\in \mathbb{N}$. Then the following statements are
        equivalent:
        \begin{enumerate}
            \item $\Ggldim(R)\leq n$;
            \item  $\pd_R(I)\leq n$, for all injective $R$-modules
            $I$.
        \end{enumerate}
\end{lem}

The next Lemma gives a characterization of quasi-Frobenius rings.

\begin{lem}[{\cite[ Theorem 1.50]{QF}}]\label{lem1.1}

       For a ring $R$, the following statements are equivalent:
\begin{enumerate}
            \item $R$ is quasi-Frobenius;

            \item $R$ is Noetherian and $Ann_R(Ann_R(I))=I$ for
            any ideal  $I$ of $R$, where $Ann_R(I)$ denotes the
            annihilator of $I$ in $R$.
              \end{enumerate}
\end{lem}

Recall that the finitistic Gorenstein projective dimension of a
ring $R$, denoted by $FGPD(R)$, is defined in \cite{HH} as
follows: $$FGPD(R):=\{\Gpd_R(M)\mid\, M \,
                R-module\; and \; \Gpd_R(M)<\infty\}.$$

\begin{proof}[Proof of Theorem~\ref{thm1}]
 (1) Suppose that
    $\Ggldim(R)=n<\infty$ for some positif integer $n$.
    If $n\geq 1$, let $I$ be  an injective $R$-module. By \cite[Lemma
            3.4]{BM2},
            $\pd_R(I)\leq n$. Then there is an exact sequence of $R$-modules
            $$0\longrightarrow
            P_n\longrightarrow P_{n-1}\longrightarrow \cdots
        \longrightarrow P_0 \longrightarrow I \longrightarrow 0$$
        with $P_i$ projective and hence free ($R$ is local).
       Since $A$ is local and $mE=0$, every  finitely generated
        ideal of $R$ has a nonzero annihilator. From \cite[Corollary 3.3.18]{Glaz}, $coker(P_n \longrightarrow
        P_{n-1})$ is flat. Then $\fd_R(I)\leq (n-1)$.
        Therefore   from \cite[Theorem 4.11]{BM2}
        and \cite[Theorem 3.14]{HH} we obtain $$\Gwdim(R)\leq n-1=\Ggldim(R)-1\qquad\qquad (\star)$$
    On the other hand,  $R$ is Noetherian by
        (\cite[Theorem 25.1]{Extension trivial}), and from \cite[Corollary
        2.3]{BM2} we get
         $$\Gwdim(R)=\Ggldim(R).\qquad\qquad (\star\star)$$
         So from $(\star)$ and $(\star\star)$ we conclude that
          $\Ggldim(R)=\infty$.

          Now if $\Ggldim(R)=0$, then $R$ is quasi-Frobenius. First we
          claim that $A$ is a quasi-Frobenius ring. Since $A$ is Noetherian
          and by Lemma~\ref{lem1.1} we must prove only that
          $Ann_A(Ann_A(I))=I$ for any ideal $I$ of $A$. Let $I$ be
          an ideal of $A$. Since $R$ is quasi-Frobenius it is
          easy to see that
          $Ann_R(Ann_R(I\ltimes E))=Ann_A(Ann_A(I))\ltimes E=I\ltimes
          E$. Hence  $I=Ann_A(Ann_A(I))$ and $A$ is quasi-Frobenius; thus $\Ggldim(A)=0$.
          On the other hand, since $R$ is quasi-Frobenius,  $R$ is
          self-injective. Then $\Ext_R^i(A,R)=0$ for any integer $i\geq
          1$ and so $\id_A(m\oplus_A
          E)=\id_R(R)=0$ by \cite[Lemma 4.35]{Extension trivial}. Hence, $m\oplus_A
          E$  is a projective $A$-module by Lemma~\ref{lem2.1}; in particular $E$ is a projective
          $A$-module and so $E$ is
          free since $A$ is local. Contradiction since $mE=0$ and $m\neq
          0$. Therefore, we conclude that $\Ggldim(R)=\infty$.

(2) First, suppose that
            $\Ggldim(R)<\infty$. Note that since $A$ is perfect, $R$ is perfect too by
                \cite[Proposition 1.15]{Extension trivial}. Combining \cite[cor 7.12]{bass} and \cite[Theorem
                2.28]{HH} we conclude that
                $FGPD(R)=FPD(R)=0$   and so $\Ggldim(R)=FGPD(R)=0$.
            Then  from Lemma~\ref{lem2.1} and \cite[Theorem 7.56]{QF}
           $R$ is quasi-Frobenius.
           In particular  $R$ is
           Noetherian and by (1) $A=K$ is a
            field. Now we claim that $dim_K E=1$. Assume that $dim_K E\geq 2$ and let $E'\subsetneq E$ be a proper
            submodule of $E$. Obviously   $0\ltimes E\subseteq Ann_R(Ann_R(0\ltimes E'))\neq 0\ltimes
            E'$, this is a contradiction since  $R$ is quasi-Frobenius  and by Lemma~\ref{lem1.1}.
            Therefore $dim_K E=1$ and $E\cong K$. Then $R=K\ltimes K$.\end{proof}

\begin{exmp}\label{exptriv1}
Let $K$ be a field, $X_{1},X_{2},...,X_{n}$  $n$ indeterminates
over $K$, $A =K[[X_{1},...,X_{n}]]$, the power series ring in $n$
variables over $K$, and $R :=A \ltimes K$. Then, $\Ggldim(R)
=\infty $.
\end{exmp}

         Next we announce the second main thm of this section.

\begin{thm}\label{thm2}
        Let $D$ be an integral domain which is not a field, $K$ its
        quotient field, $E$ a $K$-vector space, and $R:= D\ltimes E$. Then $\Ggldim(R)=\infty$
\end{thm}

To prove this thm we need the following Lemmas.

\begin{lem}[{\cite[Remarks 3.10]{BM2}}]\label{lem3.1}
                    For a
                    ring $R$, if $\Ggldim(R)$ is finite, then  \begin{eqnarray}
                    \nonumber  \Ggldim(R) &=& \sup\{\Gpd_{R}(R/I)\,\mid\,I\  ideal\ of\ R\} \\
                    \nonumber        &=& \sup\{\Gpd_{R}(M)\,\mid\,M\ finitely\ generated\ R-module\}.
                    \end{eqnarray}
\end{lem}
\begin{lem}[{\cite[Corollary 1.38]{QF}}]\label{lem3.2}
        Let $A$ be a ring. If $A$ is self-injective (i.e., $\id_A(A)=0$),
        then $Ann_A(Ann_A(I))=I$ for any finitely generated ideal
        $I$ of $A$.
\end{lem}
\begin{proof}{of Theorem~\ref{thm2}.}
            First we claim that $\frac{R}{0\ltimes
        E}$ is not a Gorenstein projective $R$-module. For this, let $0\neq a\in D$ a
        non-invertible element, then $R(0,a)=0\ltimes Da$ is an ideal
        of $R$. Clearly, $0\ltimes Da\varsubsetneq 0\ltimes E\subseteq Ann_R(Ann_R(0\ltimes
        Da))$, and by Lemma~\ref{lem3.2}, $\id_R(R)\neq 0=\id_D(E)$,
        then $\Ext_R^i(\frac{R}{0\ltimes E},R)\cong \Ext_R^i(D,R)\neq 0$, for some $i\geq 1$ \cite[Proposition 4.35]{Extension trivial}. So $\frac{R}{0\ltimes
        E}$ is not a Gorenstein projective $R$-module \cite[Proposition 2.3]{HH}.
        Now we claim that $0\ltimes E$ is not a Gorenstein projective
        $R$-module. Deny.  $0\ltimes E$ is  a Gorenstein projective
        $R$-module. Then  there is
        an exact sequence of $R$-modules $$0\longrightarrow 0\ltimes
        E\longrightarrow F\longrightarrow G\longrightarrow 0 \qquad(1)$$
        where $F\cong R^{I}$ is a free $R$-module and $G$ is
        Gorenstein projective  by \cite[Proposition 2.4]{HH}. Consider the pushout diagram

        $$
\begin{array}{ccccccccc}
   &  &  &  & 0 &  & 0 &  &  \\
   &  &  &  & \downarrow &  & \downarrow &  &  \\
  0 & \longrightarrow & 0\ltimes E & \longrightarrow & R & \longrightarrow & \frac{R}{0\ltimes E} & \longrightarrow & 0 \\
   &  & \shortparallel &  & \downarrow &  & \downarrow &  &  \\
   0& \longrightarrow & 0\ltimes E & \longrightarrow & R^{I} & \longrightarrow & C' & \longrightarrow & 0 \\
   &  &  &  & \downarrow &  & \downarrow &  &  \\
   &  &  &  & R^{I'} & = & R^{I'} &  &  \\
   &  &  &  & \downarrow &  & \downarrow &  &  \\
   &  &  &  & 0 &  & 0 &  &  \\
\end{array}$$


        Combining  the exact sequence $(1)$ and the  short exact sequence in the
        pushout $0\rightarrow 0\ltimes E \rightarrow R^I \rightarrow C'\rightarrow
        0$, yields $C'\cong G$ is Gorenstein projective. Then from  the  short exact sequence
        $0\rightarrow  \frac{R}{0\ltimes E}\rightarrow  C'\rightarrow R^{I'}\rightarrow 0$, we get $\frac{R}{0\ltimes
        E}$ is Gorenstein projective \cite[Theorem
        2.5]{HH}. But this contradicts the fact that $\frac{R}{0\ltimes E}$ is not Gorenstein projective in the first part of the proof. Then $0\ltimes E$ is not Gorenstein
        projective. On the other hand, from the short exact sequence
        $0\longrightarrow (0\ltimes E)^{J}\longrightarrow R^{J}\longrightarrow 0\ltimes E\longrightarrow
        0$ we obtain $\Gpd_R(0\ltimes
        E)=\infty$ \cite[Propsition 2.18]{HH}. Therefore $\Ggldim(R)=\infty$.\end{proof}

       Note that the condition ``$D$ is not a field''  in Theorem~\ref{thm2} is necessary. For,
the next corollary  shows that for  any field $K$,
$\Ggldim(K\ltimes K)=0$. However \cite[Lemma 2.2]{Mahdou} asserts
that $\gldim(K\ltimes K)=\infty$.

\begin{cor}\label{cor1}
       Let $K$ be a field. Then:
        \begin{enumerate}
        \item $\Ggldim(K\ltimes K)=0$.
        \item $\Ggldim(K\ltimes K^n)=\infty$, for any $n\geq 2$.
        \end{enumerate}
\end{cor}

\begin{exmp}\label{exptriv2}
Let  $R:=\Z \ltimes \Q$, where $\Z$ is the ring of integers and
$\Q$ the field of rational numbers. Then $\Ggldim(R) =\infty $.
\end{exmp}

        Next we exibit an exmp showing that, in general,
the transfer of  the notion of Gorenstein projective module does
not carry up to pullback constructions.

\begin{exmp}\label{exppulb}
Let $(D,m)$ be a discrete valuation domain and $K=\qf(D)$.
Consider the following pullback

        $$
\begin{array}{ccc}
  R=D\ltimes K & \longrightarrow & T=K\ltimes K \\
  \downarrow &  & \downarrow\\
  D\cong\frac{R}{0\ltimes K} & \longrightarrow & K \\
\end{array}$$

        Let $0\neq a\in m$ and  $I=0\ltimes Da$. Consider the following short exact sequence of
        $R$-modules
        $$0\longrightarrow 0\ltimes K \longrightarrow R\stackrel{u}\longrightarrow 0\ltimes Da\longrightarrow
        0$$
        where
       $u(b,e)=(b,e)(0,a)=(0,ba)$. Similar arguments used in the proof of Theorem~\ref{thm2} yield $0\ltimes Da$ is not
        Gorenstein projective,  $I\otimes_{R} T\cong 0\ltimes K$ is a Gorenstein
        projective ideal of $T$, and $I\otimes_{R} \frac{R}{0\ltimes K}\cong \frac{R}{0\ltimes
        K}$ is a free $\frac{R}{0\ltimes K}$-module, then Gorenstein projective.
\end{exmp}

\end{section}

\end{document}